\documentclass{amsart}[12pt]

\usepackage{amsmath}

\numberwithin{equation}{section}

\newtheorem{theorem}{Theorem}[section]
\newtheorem{lemma}[theorem]{Lemma}

\newtheorem{prop}[theorem]{Proposition}

\theoremstyle{definition}

\theoremstyle{remark}
\newtheorem*{remark}{Remark}

\begin{document}
\openup3pt

\title{The Twisted Mellin Transform}
\author{Zuoqin Wang}
\address{Department of Mathematics\\
MIT\\Cambridge, MA  02139 \\ USA}
\email{wangzq@math.mit.edu}
\date{}

\begin{abstract}
The ``twisted Mellin transform" is a slightly modified version of
the usual classical Mellin transform on $L^2([0, \infty))$. In this
short note we investigate some of its basic properties. From the
point of view of combinatorics one of its most interesting properties is that it intertwines the differential
operator, $df/dx$, with its finite difference analogue, $\nabla f=
f(x)-f(x-1)$. From the point of view of analysis one of its most
important properties is that it describes the asymptotics of one
dimensional quantum states in Bargmann quantization.
\end{abstract}

\maketitle

\section{Introduction}
\label{sec:1}

The standard Mellin Transform is defined by the formula
\begin{equation}
 \label{eq:1.1}
 Mf(s) = \int_0^\infty f(x)x^{s-1} dx.
\end{equation}
This paper has as its topic the following ``twisted" version of
(\ref{eq:1.1}):
\begin{equation}
 \label{eq:1.2}
 \mathcal{M}f(s) = \frac{\int_{0}^\infty f(x) x^s
 e^{-x}\ dx}{\int_0^\infty x^s e^{-x}\ dx}.
\end{equation}
This transform has a number of remarkable properties, the most
remarkable perhaps being that it intertwines the standard
differential operator, $\frac d{dx}$, and the finite difference
analogue of $\frac d{dx}$:
\begin{equation}
 \label{eq:1.3}
 \nabla f(x) = f(x) - f(x-1).
\end{equation}
By a theorem of Mullin and Rota it is known that there exists an
invertible operator intertwining the ``umbral" calculi generated by
$\frac d{dx}$ and $\nabla$; but, as far as we know the explicit
expression (\ref{eq:1.2}) for this intertwiner is new.

The main topic of this paper is an asymptotic formula for
(\ref{eq:1.2}) when the function, $f$, in the integrand is a
\emph{symbol} of degree $k$, i.e. has the property that for any $r
\in \mathbb N$ there exists a constant $C_r$ such that
\begin{equation}
 \label{eq:1.4}
 \left|\frac{d^rf}{dx^r} (x)\right| \le C_r x^{k-r}.
\end{equation}
More explicitly, we will show that for such functions
\begin{equation}
 \label{eq:1.5}
 \mathcal{M}f(x) \sim \sum_r f^{(r)}(x) g_r(x),
\end{equation}
where $f^{(r)}(x)=\frac {d^r}{dx^r}f(x)$, and $g_r(x)$ is a
polynomial of degree $[r/2]$ given by a simple recursion formula. In
some joint work with V. Guillemin, now in progress, we will use this
formula to obtain results about the spectral density functions of
toric varieties.

A few words about the organization of this paper. In section 2 we
will prove some elementary facts about the domain and range of
$\mathcal{M}$, derive a twisted version of the standard inversion
formula for the Mellin transform, prove that $\mathcal{M}$ has the
intertwining property that we described above and compile a table of
twisted Mellin transforms for most of the standard elementary
functions. In section 3 we will use steepest descent techniques to
derive (\ref{eq:1.5}) and give two rather different recipes for
computing the $g_r$'s, one analytic and one combinatorial. (By
comparing these two recipes we obtain some curious combinatorial
identities for the Stirling numbers of the first kind.)

We would like to thank Richard Stanley for a number of helpful
comments on the umbral calculus, Stirling numbers, and his
suggestion on combinatorial properties of the sequence of
functions $f_r$'s.

\section{The Twisted Mellin transform}

Let $\mathbb C$ be equipped with the Bargmann measure $\mu =
e^{-|z|^2}dzd\bar z.$  Given a function $f \in C^\infty(\mathbb
C)$, one would like to study the asymptotics of the spectral
measure
\begin{equation*}
T_k(f) = \mathrm{Tr}(\pi_{k} M_f \pi_{k}),
\end{equation*}
associated with the quantum eigenstate, $z^k$, as $k \to \infty$,
where $\pi_k$ is the orthogonal projection from $L^2(\mathbb C,
\mu)$ onto the one dimensional subspace spanned by $z^k$, and
$M_f$ is the operator ``multiplication by $f$". By averaging with
respect to the $\mathbb T^1$-action, we can assume $f \in
C^\infty(\mathbb C)^{\mathbb T^1}$, i.e.
\begin{equation*}
f(z) = f(r^2),
\end{equation*}
where $r=|z|$ is the modulus of complex number $z$.

For $k \in \mathbb N$, one has
\begin{equation*}
\aligned T_k(f)=\frac{\langle f z^k, z^k \rangle_\mu}{\langle z^k,
z^k\rangle_\mu} & = \frac{\int_0^\infty f(r^2) r^{2k+1} e^{-r^2}\
dr}{\int_0^\infty
r^{2k+1} e^{-r^2}\ dr}\\
& = \frac{\int_0^\infty f(x) x^{k} e^{-x}\ dx}{\int_0^\infty x^{k}
e^{-x}\ dx}.
\endaligned
\end{equation*}
So the asymptotic properties of $T_k(f)$ can be deduced from
asymptotic properties of the twisted Mellin transform
\begin{equation}
\label{eq:2.1}
 \mathcal{M}f(s) =f^{\#}(s) = \frac{\int_{0}^\infty f(x) x^s
 e^{-x}\ dx}{\int_0^\infty x^s e^{-x}\ dx}
\end{equation}
where $f \in C^\infty(\mathbb R^+)$ and $s \ge 0$. Note that the
denominator is the Gamma function $\Gamma(s+1)$, while the numerator
is just the Mellin transform of the function $xe^{-x}f(x)$.

For this integral to converge, we will assume that $f$ is of
polynomial growth, i.e.
\begin{equation}
 \label{eq:2.2}
 |f(x)| \le C x^N \mbox{\quad for\ some\ }N.
\end{equation}

Some basic properties of the transform are the following

\begin{prop}
\label{prop:2.1}

 Suppose $a, b \in \mathbb R$, $c>0$, $n \in \mathbb
 N$, $f$ is a function of polynomial growth, then

$\mathrm{(1)}$ For $g(x)=x^a f(x)$,
\begin{equation}
\label{eq:2.3}
 \mathcal{M}g(s) = \frac{\Gamma(s+a+1)} { \Gamma(s+1)} \mathcal{M}
 f(s+a),
\end{equation}
\indent \qquad and for $g(x)=e^{-cx}f(x)$,
\begin{equation}
 \label{eq:2.4}
 \mathcal{M}g(s) = (c+1)^{-s-1} \mathcal{M}f_c(s),
\end{equation}
\indent \qquad where $f_c(x)$ is the dilation, $f_c(x)=f(\frac
x{c+1})$.

$\mathrm{(2)}$ For $g(x)=\frac {df}{dx}(x)$,
\begin{equation}
 \label{eq:2.5}
 \mathcal{M}g(s)=\nabla \mathcal{M}f(s) :=
 \mathcal{M}f(s)-\mathcal{M}f(s-1),
\end{equation}
\indent \qquad and more generally, for any $n \in \mathbb N$ and
$g(x)=f^{(n)}(x)$,
\begin{equation}
 \label{eq:2.6}
 \mathcal{M}g(s)=\nabla^n(\mathcal{M}f)(s) = \sum_{i=0}^n (-1)^i {n
 \choose i}  \mathcal{M}f(s-i).
\end{equation}

$\mathrm{(3)}$ For $g(x)=f(x)\ln{x}$,
\begin{equation}
 \label{eq:2.7}
 \frac d{ds}\mathcal{M}f(s) = \mathcal{M}g(s) - \mathcal{M}f(s)
 \frac{\Gamma'(s+1)}{\Gamma(s+1)}.
\end{equation}

$\mathrm{(4)}$ For $g(x)=\int_0^x f(t)\ dt$,
\begin{equation}
 \label{eq:2.8}
 \mathcal{M}g(s)=\sum_{i=0}^{[s]-1} \mathcal{M}f(s-i) +
 \mathcal{M}g(s-[s]).
\end{equation}
\indent \qquad In particular,
\begin{equation}
 \label{eq:2.9}
 \mathcal{M} g(n)= \sum_{i=0}^n \mathcal{M}f(i);
\end{equation}
\end{prop}
\begin{proof}

The assertion (1) is obvious.

To prove (2), we note that for $g(x)=f'(x)$,
\begin{equation*} \aligned \mathcal{M}g(s) &=
\frac{\int_{0}^\infty f'(x) x^s e^{- x}\
dx}{\int_0^\infty x^s e^{- x}\ dx}\\
& = \frac{\int_0^\infty f(x)( x^s-s x^{s-1})e^{- x}\
dx}{\int_0^\infty
x^s e^{-x}\ dx}\\
& = \frac{ \int_0^\infty f(x) x^s e^{- x}\ dx}{\int_0^\infty x^s
e^{- x}\ dx} - \frac{\int_0^\infty f(x) x^{s-1} e^{- x}\
dx}{ \int_0^\infty x^{s-1} e^{- x}\ dx}\\
& =  \mathcal{M}f(s) -  \mathcal{M}f(s-1).
\endaligned
\end{equation*}
The property (\ref{eq:2.6}) is easily deduced from (\ref{eq:2.5})
by induction, and (3) is a direct computation.

To prove (4), we note that by integration by parts,
\begin{equation}
 \label{eq:2.10}
 \mathcal{M}g(s) = \mathcal{M}f(s) + \mathcal{M}g(s-1),
\end{equation}
which implies (\ref{eq:2.8}). As for (\ref{eq:2.9}), this follows
from the obvious fact $\mathcal{M}g(0)=\mathcal{M}f(0)$.
\end{proof}

From the definition its easy to see that the twisted Mellin
transform is smooth, i.e. it transform a smooth function to a
smooth function. Moreover, it transforms a function which is of
polynomial growth of degree $N$ to a function which is of
polynomial growth of degree $N$, and Schwartz functions to
Schwartz functions:

\begin{prop}
\label{prop:2.2}

$\mathrm{(1)}$ Suppose $|f(x)| \le C x^N$, then $|\mathcal{M}f(s)|
\le C' s^N.$

$\mathrm{(2)}$ $\mathcal{M}$ maps Schwartz functions to Schwartz
functions.
\end{prop}

\begin{proof}

(1) This comes from the definition:
\begin{equation*}
 |\mathcal{M}f(s)| \le \frac{\int_0^\infty C x^N x^s e^{-x}\
 dx}{\Gamma(s+1)} = C \frac{\Gamma(s+N+1)}{\Gamma(s+1)} \le C' s^N.
\end{equation*}

(2) Suppose $f$ is a Schwartz function, i.e. for any $\alpha,
\beta$, there is a constant $C_{\alpha, \beta}$ such that
$\sup_x{|x^\alpha \partial^\beta f(x)|} \le C_{\alpha, \beta}$.

For $\beta=0$, $|x^\alpha f(x)| \le C$ implies $|s^\alpha
\mathcal{M}f(s)| \le C'$.

For $\beta=1$, we apply (\ref{eq:2.7}) and the above result to get
$|s^\alpha \frac d{ds}\mathcal{M}f(s)| \le C_\alpha$.

For $\beta \ge 1$, let $\psi(s)={\Gamma'(s)}/{\Gamma(s)}$. Then by
repeated applications of (\ref{eq:2.7}) one can see that $\frac
{d^n}{ds^n}\mathcal{M}f(s)$ is a linear combination of the
functions $\mathcal{M}g_i(s)\psi^{(j)}(s+1)$, where
$g_i(x)=f(x)(\ln{x})^i$ and
\begin{equation}
 \label{eq:2.11}
 \psi^{(m)}(s+1) = \frac{d^m}{ds^m}\psi(s+1)
\end{equation}
is the polygamma function, which is bounded for each $m$, as is
clear from its integral representation:
\begin{equation*}
|\psi^{(m)}(s+1)|=\left|(-1)^{m+1}\int_0^\infty
\frac{t^me^{-(s+1)t}} {1-e^{-t}}dt\right| \le \int_0^\infty
\frac{t^me^{-t}}{1-e^{-t}}dt = \zeta(m+1)\Gamma(m+1).
\end{equation*}
Thus by induction we easily deduce that $|s^\alpha
\partial^\beta \mathcal{M}f(s)| \le C_{\alpha,\beta}$.
\end{proof}

\begin{remark}
Since the twisted Mellin transform transforms a Schwartz function to
a Schwartz function, we can define the twisted Mellin transform on
tempered distributions by duality.
\end{remark}

We will next compute the twisted Mellin transform for some
elementary functions such as polynomials, exponentials and
trigonometric functions.

(a) For $f(x)=x^a$,
\begin{equation}
 \label{eq:2.12}
 \mathcal{M}f(s)=\Gamma(s+a+1)/\Gamma(s+1).
\end{equation}
In particular, if $f(x)=x^n$, $n$ a positive integer, then
\begin{equation}
 \label{eq:2.13}
 \mathcal{M}f(s)=s^{[n]}: =(s+1)(s+2)\cdots(s+n).
\end{equation}
Thus the twisted Mellin transform of a polynomial of degree $n$ is
again a polynomial of degree $n$.

(b) Suppose $a>1$, then for $f(x)=a^{-x}$,
\begin{equation}
 \label{eq:2.14}
 \mathcal{M}f(s)=(\ln{a}+1)^{-1-s}.
\end{equation}
More generally, if $f(x)=x^b a^{-x}$, then
\begin{equation}
 \label{eq:2.15}
 \mathcal{M}f(s)= (\ln a+1)^{-1-b-s}\Gamma(s+b+1)/\Gamma(s+1).
\end{equation}

(c) For $f(x)=\frac 1{1-e^{-x}}$,
\begin{equation}
 \label{eq:2.16}
 \mathcal{M}f(s)=\zeta(s+1),
\end{equation}
and as a corollary, for the Todd function $f(x)=\frac
x{1-e^{-x}}$,
\begin{equation}
 \label{eq:2.17}
 \mathcal{M}f(s)=(s+1) \zeta(s+2).
\end{equation}

(d) For $f(x)=\ln{x}$, one gets from (\ref{eq:2.7})
\begin{equation}
 \label{eq:2.18}
 \mathcal{M}f(s)=\frac{\Gamma'(s+1)}{\Gamma(s+1)},
\end{equation}
and in general, for $f(x)=(\ln{x})^n$,
\begin{equation}
 \label{eq:2.19}
 \mathcal{M}f(s)=\frac{\Gamma^{(n)}(s+1)}{\Gamma(s+1)}.
\end{equation}

(e) For the trigonometric functions $f(x)=\sin{x}$ and
$g(x)=\cos{x}$,
\begin{equation}
 \label{eq:2.20}
 \aligned \mathcal{M}f(s) &=\frac{1}{(\sqrt
 2)^{s+1}} \sin{\frac{(s+1)\pi}{4}},\\
 \mathcal{M}g(s)&=\frac{1}{(\sqrt 2)^{s+1}} \cos{\frac{(s+1)\pi}{4}}.
 \endaligned
\end{equation}

\noindent({\sl Proof.} Let $h(x)=e^{ix}$, then
$\mathcal{M}h(s)=\frac 1{(1-i)^{s+1}}$, which gives
(\ref{eq:2.20}).)

Similarly for $f(x)=\sin(ax)$ and $g(x)=\cos(ax)$,
\begin{equation}
 \label{eq:2.21}
 \aligned
 \mathcal{M}f(s)&=(1+a^2)^{-s}\sin(s\arctan{a}), \\
 \mathcal{M}g(s)&=(1+a^2)^{-s}\cos(s\arctan{a}).
 \endaligned
\end{equation}

\noindent{\sl Some concluding remarks:}

(1) From the inversion formula for the Mellin transform, we obtain
an inversion formula for the twisted Mellin transform:
\begin{equation}
 \label{eq:2.22}
 f(x)=e^x \int_{c-\infty i}^{c+\infty
 i}\Gamma(s+1)\mathcal{M}f(s)x^{-s-1}ds.
\end{equation}
and from the Parseval formula (c.f. \cite{PaK}) a ``Parsevel-like"
formula for $\mathcal M$
\begin{equation}
 \label{eq:2.23}
 \int_0^\infty f(x)g(x) x^2 e^{-2x}\ dx = \frac 1{2i}\int_{c-\infty
 i}^{c+\infty
 i}\mathcal{M}f(1-s)\mathcal{M}g(s)\frac{s(1-s)}{\sin{\pi s}}\ ds.
\end{equation}

(2) Letting $(x)^{(n)}=x(x+1)\cdots(x+n-1)$, (\ref{eq:2.13}) becomes
$\mathcal{M}f(s-1)=(s)^{(n)}$, where $f(x)=x^n$. Also expanding
$s^{[n]}$ in terms of $s^n$, we get
\begin{equation}
 \label{eq:2.24}
 s^{[n]}=(s+1)(s+2)\cdots(s+n)=\sum_{k=0}^{n} c(n+1, k+1)s^k,
\end{equation}
where $c(n,k)$ is the signless Stirling number of first kind, c.f.
\cite{St1}. Note that both $\{x^n\}$ and $\{x^{(n)}\}$ are a basis
of the polynomial ring, thus $\mathcal{M}$ is a bijection from the
polynomial ring to itself.

(3) Formula (\ref{eq:2.5}) tells us that $\mathcal M$ conjugates
the differential operator, $\frac d{dx}$, to the backward
difference operator (\ref{eq:1.3}). In combinatorics both $\frac
d{dx}$ and the backward difference operator are ``delta"
operators, with the functions $x^n$ and $x^{(n)}$ as their
sequence of basic polynomials. Thus by a theorem of R.Mullin and
G-C.Rota (\cite{MuR}), the map $T: f(x) \mapsto \mathcal{M}f(s-1)$
is invertible and the map $S \mapsto TST^{-1}$ an automorphism of
the algebra of shift-invariant operators onto the algebra of
polynomials. Moreover, $T$ maps every sequence of basic
polynomials into a sequence of basic polynomials. Such an operator
is called an \emph{umbral} operator in the umbral calculus.

(4) If we replace the Bargmann measure, $\mu$, by the generalized
Bargmann measure $\mu_\alpha = e^{-\alpha |z|^2}\ dzd\bar z$, then
we are naturally led, by the argument at the beginning of this
section, to studying the ``$\alpha$-twisted Mellin transform"
\begin{equation}
 \label{eq:2.25}
 \mathcal{M}_\alpha f(s) = \frac{\int_0^\infty f(x) x^s e^{-\alpha
 x}\ dx }{\int_0^\infty x^s e^{-\alpha x}\ dx}.
\end{equation}
All the properties in Proposition 2.1 can be easily generalized to
$\mathcal{M}_\alpha$. Moreover, it is easy to see that
\begin{equation}
 \label{eq:2.26}
 \mathcal{M}_\alpha f_c(s) = \mathcal{M}_{\alpha/c} f(s),
\end{equation}
where $f_c(x)=f(cx)$, so the $\alpha$-twisted Mellin transform of a
function can easily be computed from the twisted Mellin transform.

(5) The twisted Mellin transform can also be defined in higher
dimension in a similar way. For the function $f(r_1, \cdots,
r_d)=f_1(r_1) \cdots f_d(r_d)$, its twisted Mellin transform is just
the product of the twisted Mellin transform of $f_1, \cdots, f_d$.

\section{The Asymptotic Expansion.}

We can rewrite the twisted Mellin transform as
\begin{equation}
 \label{eq:3.1}
 \mathcal{M}f(s)=\frac{\int_0^\infty f(x)e^{s\log{x}-x}\
 dx}{\int_0^\infty e^{s\log{x}-x}\ dx}.
\end{equation}
For the phase function $\varphi(x,s)={s\log{x}-x}$, we have
\begin{equation*}
0=\frac{\partial \varphi}{\partial x}\quad \Longrightarrow \quad
x=s,
\end{equation*}
thus the function $\varphi_s(x)=\varphi(x,s)$ has a unique critical
point at $x=s$. Moreover, this is a global maximum of
$\varphi(x,s)$, since
\begin{equation*}
 \lim_{x \to
 +\infty}\varphi(x,s)=-\infty,
\end{equation*}
and
\begin{equation*}
\frac{\partial^2
 \varphi}{\partial x^2}=-\frac{s}{x^2}<0.
\end{equation*}
Hence if $f$ is a symbol, we can apply the method of steepest
descent to both denominator and numerator to get
\begin{equation}
 \label{eq:3.2}
 \mathcal{M}f(s) \sim \sum_k g_k(s)f^{(k)}(s).
\end{equation}
To compute the functions $g_k(s)$ consider the Taylor expansion of
$f$,
\begin{equation}
 \label{eq:3.3}
 f(x) = \sum_{r=0}^\infty \frac 1{r!}f^{(r)}(s) (x-s)^r.
\end{equation}
Applying $\mathcal M$ to (3.3) with $s$ fixed we get
\begin{equation}
 \label{eq:3.4}
 \mathcal{M}f(s)=\sum_{r=0}^\infty \frac 1{r!}f^{(r)}(s)f_r(s),
\end{equation}
where
\begin{equation}
 \label{eq:3.5}
 f_r(s) = \frac{\int_0^\infty (x-s)^r x^s e^{-x} \ dx}{\int_0^\infty
x^s e^{-x} \ dx}  = \sum_{i=0}^r (-1)^{r-i}{r \choose
i}s^{[i]}s^{r-i}.
\end{equation}
For $r \le 5$ small,
\begin{equation*}
 \aligned
 f_0(s)&=f_1(s)=1,\\
 f_2(s)&=2+s,\\
 f_3(s)&=6+5s,\\
 f_4(s)&=24+26s+3s^2,\\
 f_5(s)&=120+154s+35s^2.
 \endaligned
\end{equation*}
which suggests that  $f_r(s)$ is a polynomial of degree $[r/2]$ in
$s$. We will give two separate proofs of this fact, one
combinatorial and one analytic.

\noindent{\sl Proof 1.} Putting (\ref{eq:2.24}) into
(\ref{eq:3.5}), we get
\begin{equation*}
 \aligned f_r(s)&=\sum_{i=0}^r \sum_{k=0}^i (-1)^{r-i} {r \choose i}
 c(i+1, k+1) s^{r-(i-k)}\\
 &= \sum_{j=0}^r \left(\sum_{i=j}^r (-1)^{r-i} {r \choose i} c(i+1,
 i-j+1)\right) s^{r-j}.
 \endaligned
\end{equation*}

On the other hand, by the definition of the Stirling number,
\begin{equation}
 \label{eq:3.6}
 c(n+1, n+1-k)= c(n, n-k) + n c(n, n-(k-1)).
\end{equation}
and from this recurrence relation we will show:

\begin{lemma}
\label{lemma:3.1}

There are constants $C_{l,j}$, depending only on $l,j$, such that
\begin{equation}
 \label{eq:3.7}
 c(i+1, i+1-j) = \sum_{l=j}^{2j} C_{l,j}\ (i)_l,
\end{equation}
where $(i)_l = i(i-1)\cdots(i-l+1)$ is the falling factorial.
\end{lemma}

\begin{proof}

This is true for $j=0$, since $c(n+1,n+1)=1$. Notice that
\begin{equation*}
 \sum_{i=l}^n (i)_l = l!\left({l \choose l}+{l+1 \choose l}+ \cdots
 +{n \choose l}\right)=l!{n+1 \choose l+1} = \frac 1{l+1}(n+1)_{l+1}.
\end{equation*}
Now use induction and the recurrence relation (\ref{eq:3.6}).
\end{proof}

Now suppose $2j \le r$, then the coefficients of $s^{r-j}$ in
$f_r(s)$ is
\begin{equation*}
 \aligned \sum_{i=j}^r (-1)^{r-i} {r \choose i} c(i+1, i-j+1) =
 &\sum_{i=j}^r (-1)^{r-i} {r \choose i} \sum_{l=j}^{2j} C_{l,j}\
 (i)_l\\
 = &\sum_{l=j}^{2j} C_{l,j} (r)_{l} \sum_{i=l}^r(-1)^{r-i}{r-l
 \choose i-l}\\
 = &0,
 \endaligned
\end{equation*}
which proves that $f_r$ is a polynomial of degree $[r/2]$.

\noindent{\sl Proof 2.} First we derive a recurrence relation for
$f_r(s)$. Using
\begin{equation*}
 \frac{d}{dx}(s\log{x}-x) = -\frac{x-s}x
\end{equation*}
and integration by parts we get
\begin{equation*}
 \aligned \Gamma(s+1) f_r(s) &= -\int_0^\infty e^{s\ln{x}-x} x(x-s)^{r-1}\frac{d}{dx}(s\ln{x}-x)\ dx\\
 &= \int_0^\infty e^{s\ln{x}-x}\frac{d}{dx}(x(x-s)^{r-1})\ dx\\
 &= \int_0^\infty e^{s\ln{x}-x}\frac{d}{dx}((x-s)^{r}+s(x-s)^{r-1})\ dx\\
 &= r\int_0^\infty x^s e^{-x} (x-s)^{r-1}\ dx + (r-1)s\int_0^\infty
 x^s e^{-x} (x-s)^{r-2}\ dx,
 \endaligned
\end{equation*}
i.e.
\begin{equation}
 \label{eq:3.8}
 f_r(s)=r f_{r-1}(s) + (r-1)sf_{r-2}(s).
\end{equation}
Moreover, we can compute the initial conditions directly
\begin{equation}
 \label{eq:3.9}
 f_1(s) = f_0(s)=1.
\end{equation}
\begin{remark}
The recurrence relation (\ref{eq:3.8})
also follows easily from (\ref{eq:2.3}) and (\ref{eq:2.5}). In
fact, if we denote $h_r(x)=(x-s)^r$, then
$f_r(s)=\mathcal{M}h_r(s)$, and thus
\begin{equation*}
\aligned
r f_{r-1}(s)&=\mathcal{M}h_r(s)-\mathcal{M}h_r(s-1)\\
&=f_r(s)-(\mathcal{M}(xh_{r-1})(s-1)-s\mathcal{M}h_{r-1}(s-1))\\
&=f_r(s)-s(\mathcal{M}h_{r-1}(s)-\mathcal{M}h_{r-1}(s-1))\\
&=f_r(s)-s(r-1)\mathcal{M}h_{r-2}(s).
\endaligned
\end{equation*}
\end{remark}

From (\ref{eq:3.8}), (\ref{eq:3.9}) and induction, it follows
again that $f_r(s)$ is a polynomial of degree $[r/2]$. Thus coming
back to (3.4) we have proved

\begin{theorem}
 \label{thm:3.2}

For any symbolic function $f$, we have
\begin{equation}
 \label{eq:3.10}
 \mathcal{M}f(s) \sim \sum_{r} \frac 1{r!}f^{(r)}(s)f_r(s),
\end{equation}
where $f_r(s)$ is the polynomial of integer coefficients of degree
$[r/2]$ given by (\ref{eq:3.5}).
\end{theorem}

The polynomials $f_r(s)$ have many interesting combinatorial
properties:

(1) Since $f_r(s)$ is a polynomial of degree $[r/2]$, we can write
\begin{equation}
 \label{eq:3.12}
 f_r(s)=\sum_{i=0}^{[r/2]}a_{r,i} s^i,
\end{equation}
the coefficients satisfying the recurrence relation
\begin{equation}
  \label{eq:3.13}
 a_{r,i}=r a_{r-1,i}+ (r-1) a_{r-2,i-1}
\end{equation}
and initial conditions
\begin{equation*}
 a_{r,0}=r!,\quad a_{2k,k}=(2k-1)!!,
\end{equation*}
which implies
\begin{equation*}
 \aligned
 a_{r,1}=& r! \left(\frac 1r + \frac 1{r-1} + \cdots + \frac 12\right),\\
 a_{r,2}=& r! \left(\frac{(r-1)a_{r-2,1}}{r!}+\frac{(r-2)a_{r-3,1}}{(r-1)!}+ \cdots +
 \frac{3a_{2,1}}{4!} \right),
 \endaligned
\end{equation*}
and in general
\begin{equation}
 \label{eq:3.14}
 a_{r,k}= r! \left(\frac{(r-1)a_{r-2,k-1}}{r!}+\frac{(r-2)a_{r-3,k-1}}{(r-1)!}+ \cdots +
 \frac{(2k-1)a_{2k-2,k-1}}{(2k)!}\right).
\end{equation}

(2) The coefficients, $a_{r, i}$, of $f_r(s)$, are exactly those
appeared as coefficients of polynomials used for exponential
generating functions for diagonals of unsigned Stirling numbers of
the first kind. More precisely, for fixed $k$, the exponential
generating function for the sequence $\{c(n+1, n+1-k)\}_{n \ge 0}$
is given by (c.f. sequence A112486 in ``The On-Line Encyclopedia
of Integer Sequences")
$$\sum_{n=0}^\infty c_{n+1, n+1-k} \frac{x^n}{n!} = e^x \sum_{n=k}^{2k} \left(a_{n, n-k}\frac{x^n}{n!}\right).$$

(3) The sequence of functions $f_r$'s have a pretty simple
exponential generating function:
\begin{equation*}
\aligned \sum_{r=0}^\infty f_r(s) \frac{x^r}{r!} & =
\sum_{i=0}^\infty \sum_{r=i}^\infty (-1)^{r-i}\frac 1{r!}{r
\choose
i}s^{[i]}s^{r-i}x^r \\
& = \left(\sum_{i=0}^\infty
\frac{s^{[i]}x^i}{i!}\right)\left(\sum_{r=i}^\infty
(-1)^{r-i}\frac{s^{r-i}x^{r-i}}{(r-i)!}\right) \\
& = \frac{e^{-sx}}{(1-x)^{1+s}}.
\endaligned
\end{equation*}

(4) From the generating function above we get a combinatorial
interpreting of $f_r(s)$ for integers $s$: $r!f_r(s)$ is the
number of $r \times r$ $\mathbb N$-matrices with every row and
column sum equal to $3+2s$ and with at most 2 nonzero entries in
every row. (c.f. Exercise 5.62 of \cite{St}).

(5) There are also other combinatorial interpreting for small
value of $s$. For example, the sequence $f_r(1)$ count
permutations $w$ of $\{1, 2, \cdots, r+1\}$ such that $w(i+1) \ne
w(i) + 1$ (c.f. the sequence A000255 of ``On-line Encyclopedia of
Integer Sequences"). For $s=2$, we have
\begin{equation*}
f_r(2) = \frac{2^{-r^2}}{r!} \sum_{M \in D_r}(\det M)^4,
\end{equation*}
where $D_r$ is the set of all $r \times r$ matrices of $\pm 1$'s.
(c.f. Exercise 5.64(b) of \cite{St}).

We will conclude by deriving a slight variant of the asymptotic
expansion above, which will be needed for the application in
\cite{GW}. Given a symbolic function $f$, consider the integral
\begin{equation}
 \label{eq:3.15}
 A_N(f)(s) = \frac{\int_0^\infty f(x)x^{Ns}e^{-Nx}\ dx}{\int_0^\infty
 x^{Ns}e^{-Nx}\ dx},
\end{equation}
as $N \to \infty$. By definition, this is just the ``$N$-twisted
Mellin transform" $\mathcal{M}_Nf(Ns)$, which, according to
(\ref{eq:2.26}), equals $\mathcal{M}f_N(Ns)$, where
$f_N(x)=f(x/N)$. Thus by Theorem \ref{thm:3.2},
\begin{equation}
 \label{eq:3.16}
 A_N(f)(s) \sim \sum_k \left(\frac 1N\right)^k f^{(k)}(s) g_k(Ns).
\end{equation}
Note that since $g_k(x)$ is a polynomial of degree $[k/2]$, the
above formula does give us an asymptotic expansion. In particular,
we have
\begin{equation*}
 A_N(f)(s)=f(s)+\frac 1N \left(f'(s)+f''(s)\frac s2\right)+\frac
 1{N^2}\left(f''(s)+f'''(s)\frac{5s}6+f^{(4)}(s)\frac{s^2}8\right)+O(N^{-3}).
\end{equation*}

\end{document}